\documentclass[preprint,12pt]{elsarticle}

\usepackage{amsmath, amsthm, amsrefs}
\usepackage{mathtools}

\usepackage{graphicx}
\usepackage{amssymb}

\let\today\relax 

\makeatletter   
\def\ps@pprintTitle{%
   \let\@oddhead\@empty
   \let\@evenhead\@empty
   \let\@oddfoot\@empty
   \let\@evenfoot\@oddfoot
}
\makeatother

\newcommand{\Pos}{\mathcal{P}}

\renewcommand{\ge}{\geqslant}
\renewcommand{\le}{\leqslant}

\theoremstyle{definition}
\newtheorem{definition}{Definition}
\newtheorem*{definition*}{Definition}

\newtheorem{remark}[definition]{Remark}
\newtheorem*{remark*}{Remark}

\theoremstyle{plain}
\newtheorem{conjecture}{Conjecture}
\newtheorem*{conjecture*}{Conjecture}

\newtheorem*{theorem*}{Theorem}

\newtheorem*{proposition*}{Proposition}

\newtheorem{lemma}[definition]{Lemma}
\newtheorem*{lemma*}{Lemma}

\newtheorem*{corollary*}{Corollary}

\begin{document}

\begin{frontmatter}


\title{Cyclotomic ordering conjecture}


\address[S.P. Glasby]{Center for the Mathematics of Symmetry and Computation,\\
  University of Western Australia, 35 Stirling Highway, Perth 6009, Australia}

\author{Stephen Glasby$^a$}

\begin{abstract}
  This note describes a conjecture I made (in Aachen, Sept. 2018) and some initial
  thoughts towards a solution. Given positive integers $m,n$, the conjecture
  is that either $\Phi_m(q)\le\Phi_n(q)$ or $\Phi_m(q)\ge\Phi_n(q)$ holds
  for all integers $q\ge2$. Pomerance  and Rubinstein-Salzedo proved the
  conjecture in~\cite{PR}.
\end{abstract}




\end{frontmatter}



We define a partial ordering $\preceq$ on the set $\Pos$ of positive integers.
Recall that $t^n-1=\prod_{d\mid n}\Phi_d(t)$ where the roots of the $d$th
cyclotomic polynomial $\Phi_d(t)$ are primitive roots of order $d$.
Hence $\deg(\Phi_d(t))=\phi(d)$. For $m,n\in\Pos$ 
write $m\preceq n$ if $\Phi_m(q)\le\Phi_n(q)$
for all integers $q\ge2$, and write $m\prec n$ if $m\preceq n$ and $m\ne n$.
(Clearly $a\preceq a$; $a\preceq b$ and $b\preceq a$ implies $a=b$; and
$a\preceq b$ and $b\preceq c$ implies $a\preceq c$.) Since
\[
  q-1<q+1<q^2-q+1\le q^2+1<q^2+q+1<q^4-q^3+q^2-q+1
\]
holds for all $q\ge2$, we have $1\prec2\prec6\prec 4\prec3\prec10$.
Similarly, one can show $10\prec12\prec8\prec 5\prec 14\prec 18\prec 9\prec 7\prec 15\prec 20\prec 24\prec 16\prec 30\prec 22\prec 11$.

\begin{conjecture}
  The set $\Pos$ of positive integers is totally ordered by $\prec$.
\end{conjecture}

Say that $m$ \emph{precedes} $n$ (or $n$ \emph{succeeds} $m$) if $m\prec n$ and there is no $x$
with $m\prec x \prec n$.

\begin{conjecture}
  $2\cdot3^i$ precedes $3^i$ for $i\ge2$. For $i=1$, we have $6\prec 4\prec 3$.
\end{conjecture}

Proving $\Phi_m(q)<\Phi_n(q)$ for all $q$
is the same as proving
$\Phi_n(q)-\Phi_m(q)>0$. After canceling any equal terms, this
inequality can be written $A(q)>B(q)$ where $A(t)$ and $B(t)$ are
integer polynomials whose nonzero coefficients are all positive. If the
largest nonzero coefficient is~$c$, then  $A(q)>B(q)$ holds for
all $q>c$ provided the leading monomial of $A$ is greater than
the corresponding monomial of $B$. (The base-$q$ expansion of $A(q)$ is
greater than $B(q)$.)
The conjecture asserts that the inequality also holds for $2\le q\le c$.

This reasoning will determine a putative
total ordering of $\Pos$ working for sufficiently large $q$ but maybe not
for small $q$. I wrote a program in Magma that proved that the
integers $\{1,2,\dots,2\cdot 10^4\}$ can be totally ordered. Since the
coefficients of $\Phi_n(t)$ are unbounded as $n\to\infty$, and their
maximum absolute value grows slowly, one might suspect that the conjecture is false and
the smallest incomparable
pair $(m,n)$ is large. What is positive evidence?

\begin{lemma}
  If $m,n\in\Pos$ and $\phi(m)<\phi(n)$, then $m\prec n$.
\end{lemma}

\begin{proof}
  It follows from~\cite{cH}*{Theorem~3.6} that
  $cq^{\phi(n)}<\Phi_n(q)<c^{-1}q^{\phi(n)}$ holds for all $q\ge2$ where
  $c=1-q^{-1}$. Clearly $\frac12\le c$ and $c^{-1}\le2$.
  For $n\ge3$ we know that $\phi(n)$ is even, so if $m,n\ge3$, then
  $\phi(m)\le\phi(n)-2$. Therefore
  \[
  \Phi_m(q)<c^{-1}q^{\phi(m)}\le c^{-1}q^{\phi(n)-2}\le cq^{\phi(n)}<\Phi_n(q).
  \]
  The cases when $m<3$ or $n<3$ are easily handled.
\end{proof}

Thus it suffices to consider whether distinct $m,n\in\Pos$
with $\phi(m)=\phi(n)$ are comparable, i.e. $m\prec n$ or $n\prec m$.
Clearly $\phi(m)=\phi(2m)$ if $m$ is odd.

\begin{lemma}
  If $m\in\Pos$ is odd, then $m\prec 2m$ or $2m\prec m$.
\end{lemma}

\begin{proof}
  Let $m_0$ be the radical (square-free part) of $m$. If $\mu(m_0)=1$,
  i.e. $m_0$ is a product of an even number of primes,
  then~\cite{cH}*{Theorem~3.6} implies that
  \[
    c q^{\phi(m)}<\Phi_m(q)<q^{\phi(m)}<\Phi_{2m}(q)<c^{-1}q^{\phi(m)}
    \]
    where $c=1-q^{-1}$. Similar inequalities (with $m\leftrightarrow 2m$)
    hold if $\mu(m_0)=-1$, i.e. $m_0$ is a product of an odd number of primes.
\end{proof}

\begin{remark}
  The sequence $1,2,6,4,3,10,12,8,5,14,\dots$ is A206225 in the OEIS.
  It tacitly assumes (without proof) that $\prec$ is a total ordering.
\end{remark}

\begin{remark}
  If $m\ne n$ and $\phi(m)=\phi(n)$, then $\Phi_m(t)-\Phi_n(t)$ is a
  power of $t$ times a self-reciprocal polynomial. Hence $\Phi_m(t)-\Phi_n(t)>0$ for
  $t\ge2$ implies $\Phi_m(t)-\Phi_n(t)>0$ for $0<t\le\frac12$.
\end{remark}

\end{document}